\documentclass{amsart}

\usepackage[latin1]{inputenc}
\usepackage[T1]{fontenc}
\usepackage{amsmath,amsfonts,amssymb,amsthm}
\usepackage{mathtools}

\numberwithin{equation}{section}

\newtheorem{theorem}{Theorem}
\newtheorem{corollary}{Corollary}
\newtheorem{lemma}{Lemma}

\renewcommand{\leq}{\leqslant}
\renewcommand{\geq}{\geqslant}

\DeclareMathOperator{\Res}{Res}
\DeclareMathOperator{\Disc}{Disc}

\newcommand{\N}{\mathbb{N}}
\newcommand{\Z}{\mathbb{Z}}

\newcommand{\eps}{\varepsilon}

\begin{document}

\title{Nair-Tenenbaum bounds uniform with respect to the discriminant}

\author{Kevin Henriot}

\date{}

\maketitle

\begin{abstract}
  For a suitable arithmetic function $F$ and
  polynomials ${Q_{1},\dotsc,Q_{k}}$ in  $\Z[X]$,
  Nair and Tenenbaum obtained an upper bound on
  the short sum
  $\sum_{x < n \leq x+y} F\big(|Q_{1}(n)|,\dotsc,|Q_{k}(n)|\big)$
  with an implicit dependency on the discriminant of $Q_{1} \dotsm Q_{k}$.
  We obtain a similar upper bound uniform in the discriminant.
\end{abstract}

\section{Introduction}

Let $\mathcal{M}$ denote the class of multiplicative functions $f$ such that

\begin{enumerate}
\item there exists $A \geq 1$ such that $f(p^{\ell}) \leq A^{\ell}$
  for any prime $p$ and any $\ell \in \N$,
\item for all $\eps > 0$ there exists $B=B(\eps) > 0$
such that $f(n) \leq B n^{\eps}$ for any $n \in \N$.
\end{enumerate}

Let also $\alpha,\beta \in ]0,1[$.
For $f \in \mathcal{M}$ and $(a,q)=1$, Shiu \cite{shiu} proved that
\begin{equation*}
  \sum_{\substack{x < n \leq x+y \\ n \equiv a \bmod q}} f(n)
  \ll \frac{y}{\varphi(q)} \frac{1}{\log x}
  \exp\bigg(\sum_{\substack{p \leq x \\ p \nmid q}} \frac{f(p)}{p} \bigg)
\end{equation*}
in the range $q < y^{1-\beta}$, $x^{\alpha} \leq y \leq x$,
where the implicit constant depends on $A$, $B$, $\alpha$, $\beta$. Shiu's result
in \cite{shiu} is actually stated in a slightly different way, which
is however easily seen to be equivalent to the above.
This was the first bound of this generality on sums of multiplicative
functions on large subsequences of the integers, that is on arithmetic
progressions in this case, and it proved to be very useful
for different applications.

Nair \cite{nair} generalized Shiu's work to sums of the type $\sum_{n \leq x} f\big(|Q(n)|\big)$
with ${f \in \mathcal{M}}$ and $Q \in \Z[X]$. Nair and Tenenbaum \cite{nairtenenbaum} further
generalized Nair's result to functions of several variables satisfying a
property weaker than submultiplicativity. We quote their main result
here. For fixed constants $k \geq 1$, $A \geq 1$, $B \geq 1$,
$\eps > 0$, let $\mathcal{M}_{k}(A,B,\eps)$ be the class
of non-negative functions $F$ of $k$ variables such that
\begin{equation*}
  F(a_{1}b_{1},\dotsc,a_{k}b_{k}) \leq
  \min\big(A^{\Omega(a_{1} \dotsm a_{k})},B(a_{1} \dotsm a_{k})^{\eps}\big) F(b_{1},\dotsc,b_{k})
\end{equation*}
for all $a_{i}$, $b_{j}$ such that $(a_{1} \dotsm a_{k} , b_{1} \dotsm b_{k}) = 1$.
\begin{theorem}[Nair and Tenenbaum]
  \label{thm:nairtenenbaum}
  Let $k \geq 1$.
  Let $Q_{1},\dotsc,Q_{k} \in \Z[X]$ be $k$ pairwise coprime and irreducible
  polynomials.
  Let $Q = Q_{1} \dotsm Q_{k}$ and denote by $g$ its degree and
  $D$ its discriminant.
  Let $\rho_{Q_{j}}(n)$ (resp. $\rho(n)$) denote
  the number of zeroes of $Q_{j}$ (resp. $Q$) modulo $n$ for $1 \leq j \leq k$.
  Assume $Q$ has no fixed prime divisor.
  Let $0 < \alpha < 1$, $0 < \delta < 1$, $A \geq 1$ and $B \geq 1$.
  Let $\eps \leq \frac{\alpha\delta}{12g^{2}}$ and $F \in \mathcal{M}_{k}(A,B,\eps)$. We have,
  uniformly in  $x \geq c_{0}\| Q\|^{\delta}$ and $x^{\alpha} < y \leq x$,
  \begin{equation}
    \label{eq:nairtenenbaum}
    \begin{split}
      \sum_{x < n \leq x+y} F\big(&|Q_{1}(n)|,\dotsc,|Q_{k}(n)|\big) \\
      &\ll y \prod_{p \leq x} \Big( 1 - \frac{\rho(p)}{p} \Big)
      \sum_{n_{1} \dotsm n_{k} \leq x} F(n_{1},\dotsc,n_{k})
      \frac{\rho_{Q_{1}}(n_{1}) \dotsm \rho_{Q_{k}}(n_{k})}{n_{1} \dotsm n_{k}},
    \end{split}
  \end{equation}
  where $c_{0}$ depends at most on $g$, $\alpha$, $\delta$, $A$, $B$ and the implicit constant
  depends at most on $g$, $D$, $\alpha$, $\delta$, $A$, $B$.
\end{theorem}
Actually, Nair and Tenenbaum do not require the polynomials
$Q_j$ to be irreducible and pairwise coprime : this assumption
is made here merely to simplify the statement of their result.
Note that the implicit constant in \eqref{eq:nairtenenbaum}
is allowed to depend on $D$.
As a consequence of its generality, Nair and Tenenbaum's
theorem can be extended to sums over integers $n$ in arithmetic progressions
and to sums over primes $p$, as shown in \cite{nairtenenbaum}.

Daniel \cite{daniel} obtained bounds of the type of
\eqref{eq:nairtenenbaum} with uniformity in the discriminant $D$.
In this article we are interested in obtaining
such bounds and we improve on Daniel's results in
several aspects, as we shall see later.

We first explain the motivation for our work.
The need for bounds of type \eqref{eq:nairtenenbaum} uniform with
respect to the discriminant of $Q$ has emerged
in the context of several number-theoretic problems.
One of these is the recent proof \cite{holosound} of Quantum Unique Ergodicity
by Soundararajan and Holowinsky, which combines different approaches
by its two authors.
Holowinsky's approach \cite{holo} relies on estimates for
shifted convolution sums $\sum_{n \leq x} \lambda_{f}(n)\overline{\lambda_{f}(n+\ell)}$,
where $\lambda_{f}$ are the renormalized Hecke eigenvalues of a Hecke eigencuspform $f$.
These sums are averaged over $|\ell| \leq x$ in the
course of Holowinsky's computations, therefore it is crucial
to obtain an estimate uniform in $\Disc\big(X(X+\ell)\big)=\ell^{2}$.
The bound used by Holowinsky in \cite{holo} is the following, where
we let $\tau_{m}$ denote the $m$-th divisor function and $\tau = \tau_{2}$.
\begin{theorem}[Holowinsky]
  \label{thm:holowinsky}
  Let $\lambda_{1}$ and $\lambda_{2}$ be multiplicative functions such that the bound
  ${|\lambda_{i}(n)| \leq \tau_{m}(n)}$ holds for some $m$.
  Let $0 < \eps < 1$, then for $x \geq c_{0}$ and uniformly in $1 \leq |\ell| \leq x$,
  \begin{equation*}
    \sum_{n \leq x} |\lambda_{1}(n)\lambda_{2}(n+\ell)|
    \ll \tau(|\ell|) \frac{x}{(\log x)^{2-\eps}}
    \prod_{p \leq x} \Big( 1 + \frac{|\lambda_{1}(p)|}{p} \Big)
    \Big( 1 + \frac{|\lambda_{2}(p)|}{p} \Big),
  \end{equation*}
  where $c_{0}$ and the implicit constant depend on $\eps$ and $m$
  at most.
\end{theorem}
Holowinsky's proof of the above result is based on the Large Sieve.
Our results presented in this paper provide an independant
proof of this theorem, together with a few refinements : $\tau(|\ell|)$
is replaced by a function $\Delta(\ell)$ with mean value $1$
and the exponent $\eps$ is removed.
Another problem to feature discriminant-uniform bounds is
the divisor problem for binary forms of degree $4$ studied by Browning
and de la Bretèche in \cite{pbdiviseurs}. Their argument relies,
among other things, on finding estimates \cite{sumbinaryforms} for the sums
\begin{equation*}
  \sum_{n_{1} \leq X_{1}} \sum_{n_{2} \leq X_{2}} f\big(F(n_{1},n_{2})\big)
\end{equation*}
where $f \in \mathcal{M}$ and $F$ is a binary form with non-zero
discriminant. Their idea is to first study the inner sum with
$n_{1}$ fixed so that $F(n_{1},n_{2})$ is a polynomial in $n_{2}$.
For this sum they use an analogue of \eqref{eq:nairtenenbaum}
(in the case $k=1$) with uniformity in the discriminant.
Here again the uniformity is essential to average over $n_{1}$.
Our results also apply in this case.

As stated above, the aim of this paper is to obtain discriminant-uniform
bounds in the setting of Nair and Tenenbaum \cite{nairtenenbaum}.
We now introduce our main result. We restrict to the case
of irreducible pairwise coprime polynomials $Q_{i}$ and multiplicative $F$
to simplify the exposition.
\begin{theorem}
  \label{thm:introdiscbound}
  Under the assumptions of Theorem~\ref{thm:nairtenenbaum},
  and assuming further that $F$ is multiplicative and that
  $\eps \leq \frac{\alpha}{50g(g+1/\delta)}$, we have,
  uniformly in $x \geq c_{0}\| Q\|^{\delta}$ and $x^{\alpha} < y \leq x$, 
  \begin{equation}
    \label{eq:mythm}
    \begin{split}
      \sum_{x < n \leq x+y} &F\big(|Q_{1}(n)|,\dotsc,|Q_{k}(n)|\big) \\
      &\ll \Delta_{D} y \prod_{p \leq x} \Big( 1 - \frac{\rho(p)}{p} \Big)
      \sum_{\substack{n_{1} \dotsm n_{k} \leq x \\ (n_{1} \dotsm n_{k},D)=1 }} F(n_{1},\dotsc,n_{k})
      \frac{\rho_{Q_{1}}(n_{1}) \dotsm \rho_{Q_{k}}(n_{k})}{n_{1} \dotsm n_{k}}
    \end{split}
  \end{equation}
  where
  \begin{equation}
    \label{eq:introdeltaD}
    \Delta_{D} = \prod_{p | D} \bigg( 1 + \sum_{\substack{\nu_{j} \leq \deg(Q_{h}) \\ (1 \leq j \leq k)}}
    F(p^{\nu_{1}},\dotsc,p^{\nu_{k}})
    \frac{\# \{ n \bmod p^{\max_{j}(\nu_{j})+1} : p^{\nu_{j}} || Q_{j}(n) \; \forall j \} }{p^{\max_{j}(\nu_{j})+1}}
    \bigg) \text{.}
  \end{equation}
  The implicit constant and $c_{0}$ depend at most on
  $g$, $\alpha$, $\delta$, $A$, $B$.
\end{theorem}
Daniel \cite{daniel} obtains a bound analogous to \eqref{eq:mythm},
with a method of proof different from us.
However instead of $\Delta_{D}$, Daniel uses the weaker term $\tilde{\Delta}_{D}$
defined as $\Delta_{D}$ in \eqref{eq:introdeltaD} but where the conditions
$p^{\nu_{j}}||Q(n)$ are replaced by $p^{\nu_{j}}|Q(n)$ ($1 \leq j \leq k$).
In the case $k=1$ we have
\begin{equation}
  \label{eq:deltaD}
  \Delta_{D} = \prod_{p|D} \bigg( 1 + \sum_{\nu \leq g} F(p^{\nu})
  \Big( \frac{\rho(p^{\nu})}{p^{\nu}} -
  \frac{\rho(p^{\nu+1})}{p^{\nu+1}} \Big) \bigg)
  \leq \tilde{\Delta}_{D} =
  \prod_{p | D} \Big( 1 + \sum_{\nu \leq g} F(p^{\nu}) \frac{\rho(p^{\nu})}{p^{\nu}} \Big)
\end{equation}
which shows that the term $\Delta_{D}$ has the advantage of taking into
account certain cancellations between values of the $\rho$ function.
With this improved term $\Delta_{D}$, we are then able to show that the
bound \eqref{eq:mythm} is best possible in the sense
that for all polynomials $Q_{i}$ and all constants $\alpha$, $\delta$, $A$, $B$, $\eps$,
it is attained for a large family of functions $F \in \mathcal{M}(A,B,\eps)$.
Our results are perhaps easier to apprehend in the setting of Shiu,
in which they take the following form.
\begin{theorem}
  \label{thm:introk=1bound}
  Let $f \in \mathcal{M}$ and $Q \in \Z[X]$. Assume $Q$ is irreducible and denote
  by $g$ its degree and $D$ its discriminant.
  Let $0 < \alpha < 1$ and $0 < \delta < 1$.
  We have, uniformly in $x \geq c_{0} \| Q\|^{\delta}$ and $x^{\alpha} < y \leq x$,
  \begin{equation*}
    \sum_{x < n \leq x+y} f\big(|Q(n)|) \ll
    \Delta_{D} y \prod_{g < p \leq x} \bigg( 1 - \frac{\rho(p)}{p} \Big)
    \exp\bigg( \sum_{\substack{p \leq x \\ p \nmid D}} \frac{f(p)}{p} \bigg)
  \end{equation*}
  where the implicit constant and $c_{0}$ depend
  on $\alpha$, $\delta$, $A$, $B$ at most, and
  where $\Delta_{D}$ is defined by \eqref{eq:deltaD} (with $F=f$).
\end{theorem}

In our article we use the method of proof of
Nair and Tenenbaum in \cite{nairtenenbaum}.
To address the issue of preserving the uniformity in the discriminant,
we employ the following bound by Stewart \cite{stewart}.
For all primes $p$, we have
\begin{align*}
  &\phantom{(\nu \geq 1)} &\rho(p^{\nu}) &\leq g p^{[\nu - \frac{\nu}{g}]} & &(\nu \geq 1)
  \text{.}
\end{align*}
This allows us to bound the key quantity $\frac{\rho(p^{\nu})}{p^{\nu}}$ by a negative
power of $p^{\nu}$, whereas classical bounds by Nagell would only
allows us to bound this quantity by a positive power of $p^{\nu}$ for
$p|D$ and large $D$.
Note that this idea was already present in the work of Daniel \cite{daniel}.

The article is organized as follows. Section~\ref{sec:notations} is
devoted to introducing the necessary notations. In Section~\ref{sec:results}
we state all of our results and we derive the theorems exposed in
the introduction from them.
In Section~\ref{sec:technicallemmas} we prove some technical lemmas
that are of constant use in our argument, and in
Sections~\ref{sec:mainbound}, \ref{sec:corollaries}, \ref{sec:lowerbound}
we prove our results.

\textbf{Acknowledgements.}  
I am very grateful to Régis de la Bretèche
for suggesting this problem to me in the first place and
for his guidance throughout the making of this paper. I would
also like to thank Tim Browning and Gérald Tenenbaum for
helpful suggestions. The research and writing of this work
was carried during an internship of the author at the
\textit{Université Paris 7 Denis Diderot} whose
hospitality is gratefully acknowledged.

\section{Notations and definitions}
\label{sec:notations}

We follow the notations of Nair and Tenenbaum in \cite{nairtenenbaum}.

\textit{On integers.}
We let $P^{+}(n)$, $P^{-}(n)$ respectively denote the greatest and the least
prime factor of an integer $n$, with the convention that $P^{+}(1)=1$
and $P^{-}(1)=\infty$. We also let $[n]$ denote the greatest integer less than
or equal to $n$.

We denote by $\Omega(n)$, $\omega(n)$ the number of prime factors of $n$,
counted respectively with or without multiplicity. We write
$\varphi(n)$ for Euler's function and $\kappa(n)$ for the squarefree
kernel of $n$, that is $\kappa(n)=\prod_{p|n}p$.

For $n,m\in \N$ we let $n|m^{\infty}$ indicate that all prime
factors of $n$ divide $m$. The notation $a||b$ means
that $a|b$ and $(a,\frac{b}{a})=1$.

\textit{On polynomials.}
For any $P \in \Z[X]$ we define $\| P \|$ as the sum of the coefficients
of $P$ taken in absolute value, and we say that $p$ is a
fixed prime divisor of $P$ when $p|Q(n)$ for all $n \in \N$.

For polynomials $Q_{1},\dotsc,Q_{k} \in \Z[X]$ we define $Q \coloneqq \prod_{j=1}^{k} Q_{j}$.
We denote by $g$ the degree of $Q$, $r$ its number of irreducible factors and
$D$ its discriminant.
We assume that $Q$ is primitive, that is that the greatest common divisor of
its coefficients is $1$.

We write the decomposition of these polynomials in irreducible factors as
\begin{align}
  \label{eq:gammah}
  Q &= R_{1}^{\gamma_{1}} \dotsm R_{r}^{\gamma_{r}}, \\
  \label{eq:gammajh} Q_{j} &= R_{1}^{\gamma_{j1}} \dotsm R_{r}^{\gamma_{jr}}
\end{align}
for $1 \leq j \leq k$.
We define $Q^{*} \coloneqq R_{1} \dotsm R_{r}$ and denote by $g^{*}$
its degree. We will mainly work with the polynomial $Q^{*}$ as it has
the important property of having a non-zero discriminant,
which we denote by $D^{*}$.
For any polynomial $P \in \Z[X]$, we let $\rho_{P}(n)$ denote the number
of zeroes of $P$ modulo $n$. We let
\begin{equation*}
  \rho \coloneqq \rho_{Q} \text{,\quad\quad}
  \rho^{*} \coloneqq \rho_{Q^{*}} \text{.}
\end{equation*}
We next recall some well-known bounds (see e.g. \cite{nagell})
on $\rho$ and $\rho^{*}$. For all primes $p$ we have
\begin{align}
  \label{eq:rhopbound}
  & & \rho(p) &\leq g, & & \\
  \label{eq:rhotrivialbound}
  &\phantom{(\nu \geq 1)} &\rho^{*}(p^{\nu}) &\leq g^{*} p^{\nu-1} & &(\nu \geq 1),\\
  \label{eq:rhocoprimetodiscbound}
  & &\rho^{*}(p^{\nu}) &= \rho^{*}(p) \leq g^{*} &
  &(p \nmid D^{*},\: \nu \geq 1) \text{.}
\end{align}
In our article we use in an essential way the following bounds by
Stewart \cite{stewart}. For all primes $p$, we have
\begin{align}
  \label{eq:rhostewartbound}
  &\phantom{(\nu \geq 1)} &\rho^{*}(p^{\nu}) &\leq g^{*} p^{[\nu - \frac{\nu}{g^{*}}]}
  \leq g p^{[\nu - \frac{\nu}{g}]} & &(\nu \geq 1) \\
  \label{eq:rhohstewartbound}
  &\phantom{(\nu \geq 1)} &\rho_{R_{h}}(p^{\nu}) &\leq \mu_{h} p^{[\nu - \frac{\nu}{\mu_{h}}]} & &(\nu \geq 1,\; 1 \leq h \leq r)
\end{align}
where $\mu_{h} = \deg(R_{h})$. Finally we let
\begin{equation}
  \label{eq:rhoR}
  \hat{\rho}_{\mathbf{R}}(n_{1},\dotsc,n_{r}) = \# \{ n \bmod [n_{1}\kappa(n_{1}),\dotsc,n_{r}\kappa(n_{r})]
  : n_{h}||R_{h}(n) \text{ for } 1 \leq h \leq r \} \text{.}
\end{equation}
It is a multiplicative function.
We record here an useful bound on $\hat{\rho}_{\mathbf{R}}$.
\begin{equation}
  \label{eq:rhoRbound}
  \frac{\hat{\rho}_{\mathbf{R}}(n_{1},\dotsc,n_{r})}{[n_{1}\kappa(n_{1}),\dotsc,n_{r}\kappa(n_{r})]}
  \leq \frac{\rho^{*}(n_{1} \dotsm n_{r})}{n_{1} \dotsm n_{r}} \text{.}
\end{equation}
To see \eqref{eq:rhoRbound}, note that
\begin{align*}
  \frac{\hat{\rho}_{\mathbf{R}}(n_{1},\dotsc,n_{r})}{[n_{1}\kappa(n_{1}),\dotsc,n_{r}\kappa(n_{r})]}
  &\leq \frac{\# \{n \bmod [n_{1},\dotsc,n_{r}] : n_{h} | R_{h}(n) \; (1 \leq h \leq r) \} }{[n_{1},\dotsc,n_{r}]} \\
  &= \frac{\# \{n \bmod n_{1} \dotsm n_{r} : n_{h} | R_{h}(n) \; (1 \leq h \leq r) \} }{n_{1} \dotsm n_{r}} \\
  &\leq \frac{\# \{n \bmod n_{1} \dotsm n_{r} : n_{1} \dotsm n_{r} | Q^{*}(n) \} }{n_{1} \dotsm n_{r}}
  \text{.}
\end{align*}

\textit{On arithmetic functions.}
Let $H$ be a function of $s$ integer variables.
We say that $H$ is submultiplicative (resp. multiplicative) if
\begin{equation*}
  H(a_{1}b_{1},\dotsc,a_{s} b_{s}) \leq H(a_{1},\dotsc,a_{s}) H(b_{1},\dotsc,b_{s})
\end{equation*}
(resp. with equality in the above) for all $a_{i}$, $b_{j}$
such that $(a_{1} \dotsm a_{s} , b_{1} \dotsm b_{s}) = 1$.
We also define, for $1 \leq j \leq s$,
\begin{equation*}
  H^{(j)}(n) = H( 1 , \dotsc , n , \dotsc , 1 )
\end{equation*}
where $n$ is at the $j$-th place.

We let $\mathcal{M}_{k}(A,B,\eps)$ be the class of non-negative functions $F$ of $k$
integer variables satisfying
\begin{equation}
  \label{eq:Fsubm}
  F(a_{1}b_{1},\dotsc,a_{k}b_{k}) \leq
  \min\big(A^{\Omega(a_{1} \dotsm a_{k})},B(a_{1} \dotsm a_{k})^{\eps}\big) F(b_{1},\dotsc,b_{k})
\end{equation}
for all $a_{i}$, $b_{j}$ such that $(a_{1} \dotsm a_{k},b_{1} \dotsm b_{k}) = 1$.
Nair and Tenenbaum \cite{nairtenenbaum}
actually consider functions $F$ satisfying the above property for
all $a_{i}$, $b_{j}$ such that $(a_{i},b_{i}) = 1$, although the proof of their
theorem requires this property only for integers $a_{i}$, $b_{j}$ such that
$(a_{1} \dotsm a_{k},b_{1} \dotsm b_{k}) = 1$.
We thus took the liberty of using the same notation as in
\cite{nairtenenbaum} to denote our slightly larger class of functions.
We remark here that $F$ is zero if $F(1,\dotsc,1)=0$.

For a function $F$ of $k$ variables such that $F(1,\dotsc,1) \neq 0$,
we define an associated minimal function
\begin{equation}
  \label{eq:Gdef}
  G(a_{1},\dotsc,a_{k}) = \max_{\substack{ b_{1},\dotsc,b_{k} \geq 1 \\ (a_{1} \dotsm a_{k},b_{1} \dotsm b_{k}) = 1 \\ F(b_{1},\dots,b_{k}) \neq 0 }}
  \frac{F(a_{1}b_{1},\dotsc,a_{k}b_{k})}{F(b_{1},\dotsc,b_{k})} \text{.}
\end{equation}
Note that $G=F$ when $F$ is multiplicative. When $F \in \mathcal{M}_{k}(A,B,\eps)$,
it is easily checked that $G$ is submultiplicative and
\begin{align}
  \label{eq:preG2}
  G(n_{1},\dotsc,n_{k}) &\leq \min\big(A^{\Omega(n_{1} \dotsm n_{k})},B(n_{1} \dotsm n_{k})^{\eps}\big) \text{,} \\
  \label{eq:preG3}
  G(n_{1},\dotsc,n_{k}) &\leq \prod_{p^{\nu}||n_{1} \dotsm n_{k}} \min(A^{\nu},Bp^{\eps\nu}) \text{.}
\end{align}

\textit{Special notation.} We let $F$ be a function of $k$ variables
such that $F(1,\dotsc,1) \neq 0$.
Decomposing polynomials $Q_{j}$ ($1 \leq j \leq k$) as in \eqref{eq:gammajh},
we remark that
\begin{align}
  \label{eq:switchtoFtilde}
  &\phantom{(n \geq 1)} & F\big(|Q_{1}(n)|,\dotsc,|Q_{k}(n)|\big)
  &= \tilde{F}\big(|R_{1}(n)|,\dotsc,|R_{r}(n)|\big) & &(n \geq 1)
\end{align}
where $\tilde{F}$ is defined by
\begin{equation*}
  \tilde{F}(n_{1},\dotsc,n_{r}) \coloneqq
  F(n_{1}^{\gamma_{11}} \dotsm n_{r}^{\gamma_{1r}},\dotsc,n_{1}^{\gamma_{k1}} \dotsm n_{r}^{\gamma_{kr}}) \text{.}
\end{equation*}

If $G$ is the minimal function associated to $F$ by \eqref{eq:Gdef},
then $\tilde{G}$ is the minimal function associated to $\tilde{F}$
in a similar fashion. Therefore
\begin{equation}
  \label{eq:Ftildesubm}
  \tilde{F}(a_{1}b_{1},\dotsc,a_{r}b_{r}) \leq
  \tilde{G}(a_{1},\dotsc,a_{r}) \tilde{F}(b_{1},\dotsc,b_{r})
\end{equation}
for all $a_{i}$, $b_{j}$ such that $(a_{1} \dotsm a_{r},b_{1} \dotsm b_{r}) = 1$.
When $F \in \mathcal{M}_{k}(A,B,\eps)$ we obviously have
$\tilde{F} \in \mathcal{M}_{r}(A^{g},B,g\eps)$ and therefore by
\eqref{eq:preG2} and \eqref{eq:preG3} we have
\begin{align}
  \label{eq:G1}
  \tilde{G}(n_{1},\dotsc,n_{r}) &\leq A^{g\Omega(n_{1} \dotsm n_{r})} \text{,} \\
  \label{eq:G3}
  \tilde{G}(n_{1},\dotsc,n_{r}) &\leq \prod_{p^{\nu}||n_{1} \dotsm n_{r}} \min(A^{g\nu},Bp^{g\eps\nu}) \text{.}
\end{align}

\section{Results}
\label{sec:results}

Our main theorem is the following.

\begin{theorem}
  \label{thm:mainbound}
  Let $k$ be a positive integer and let $Q_{j} \in \Z[X]$ ($1 \leq j \leq k$).
  Let $Q = \prod_{j = 1}^{k} Q_{j}$ and assume that $Q$ is primitive.
  Let \eqref{eq:gammah} be the decomposition of $Q$
  in irreducible factors and define $g$, $\rho$, $\hat{\rho}_{\mathbf{R}}$
  as in the previous section.
  Let $0 < \alpha < 1$, $0 < \delta < 1$, $A \geq 1$ and $B \geq 1$.
  Let also $0 < \eps < \frac{\alpha}{50g(g+\frac{1}{\delta})}$
  and $F \in \mathcal{M}_{k}(A,B,\eps)$.
  Then we have, uniformly in $x \geq c_{0} \| Q \|^{\delta}$ and $x^{\alpha} \leq y \leq x$,
  \begin{equation}
    \label{eq:mainbound}
    \begin{split}
      \sum_{x < n \leq x+y} F\big( &|Q_{1}(n)| , \dotsc , |Q_{k}(n)| \big) \\
      &\ll y \prod_{g < p \leq x} \Big( 1 - \frac{\rho(p)}{p} \Big)
      \sum_{n_{1} \dotsm n_{r} \leq x} \tilde{F}(n_{1},\dotsc,n_{r})
      \frac{ \hat{\rho}_{\mathbf{R}}(n_{1},\dotsc,n_{r}) }{ [n_{1} \kappa(n_{1}), \dotsc, n_{r} \kappa(n_{r})] }
    \end{split}
  \end{equation}
  where $c_{0}$ and the implicit constant depend at most
  on $g$, $\alpha$, $\delta$, $A$, $B$.
\end{theorem}

We also provide a bound in which the dependency on the
discriminant $D^{*}$ is made explicit.

\begin{corollary}
  \label{cor:discbound}
  Under the assumptions of Theorem~\ref{thm:mainbound},
  \begin{align*}
    \sum_{x < n \leq x+y} &F\big(|Q_{1}(n)|,\dotsc,|Q_{k}(n)|\big) \\
    &\ll \Delta_{D^{*}} y \prod_{g < p \leq x}
    \Big( 1 - \frac{\rho(p)}{p} \Big)
    \sum_{\substack{ n_{1} \dotsm n_{r} \leq x  \\ (n_{1} \dotsm n_{r},D^{*}) =1 \\ (n_{i},n_{j}) = 1 (i \neq j)}}
    \tilde{F}(n_{1},\dotsc,n_{r}) \frac{\rho_{R_{1}}(n_{1}) \dotsm \rho_{R_{r}}(n_{r})}{n_{1} \dotsm n_{r}}
  \end{align*}
  where
  \begin{equation}
    \label{eq:deltaDstar}
    \Delta_{D^{*}} = \prod_{p|D^{*}} 
    \bigg( 1 + \sum_{\substack{\nu_{h} \leq \deg(R_{h}) \\ (1 \leq h \leq r)}}
    \tilde{G}(p^{\nu_{1}},\dotsm,p^{\nu_{r}})
    \frac{\hat{\rho}_{\mathbf{R}}(p^{\nu_{1}},\dotsc,p^{\nu_{r}})}{p^{\max(\nu_{h})+1}}
    \bigg) \text{.}
  \end{equation}
  The dependencies of the various constants are as described
  in Theorem~\ref{thm:mainbound}.
\end{corollary}

\textbf{Remark.}
Using \eqref{eq:rhoRbound} and the trivial bound
\eqref{eq:rhotrivialbound} on $\rho^{*}$, we see that
\begin{equation*}
  1 \leq \Delta_{D^{*}} \leq \prod_{p|D^{*}}\Big(1+\frac{1}{p}\Big)^C
\end{equation*}
with $C = g \cdot \max_{p} \sum_{\nu_{h} \leq \deg(R_{h}) \; (1 \leq h \leq r)} \tilde{G}(p^{\nu_{1}},\dotsc,p^{\nu_{r}})$.
Therefore $\Delta_{D^{*}}$ has mean value one when averaged over $D^{*}$.

\begin{corollary}
  \label{cor:Fmultbound}
  Under the assumptions of Theorem~\ref{thm:mainbound},
  \begin{align*}
    \sum_{x < n \leq x+y} F\big( &|Q_{1}(n)| , \dotsc , |Q_{k}(n)| \big) \\
    &\ll \Delta_{D^{*}} y \prod_{g < p \leq x} \Big( 1 - \frac{\rho(p)}{p} \Big)
    \prod_{\substack{p \leq x \\ p \nmid D^{*}}} \mspace{4.0mu} \prod_{h = 1}^{r}
    \Big( 1 + \tilde{G}^{(h)}(p) \frac{\rho_{R_{h}}(p)}{p} \Big)
  \end{align*}
  where $\Delta_{D^{*}}$ is defined by \eqref{eq:deltaDstar}.
  The dependencies of the various constants are as described
  in Theorem~\ref{thm:mainbound}.
\end{corollary}

This corollary sheds some light on the difference of behavior
between the part of the sum that depends on $D^{*}$ and
the part that is independant of $D^{*}$.
Indeed for primes $p \nmid D^{*}$,
only the values $\tilde{G}(1,\dotsc,p,\dotsc,1)$,
where $p$ is at the $h$-th place ($1 \leq h \leq r$),
are involved in the bound,
whereas for primes $p|D^{*}$ we have to take into account the
values $\tilde{G}(p^{\nu_{1}},\dotsc,p^{\nu_{r}})$ for $\nu_{h} \leq \deg(R_{h})$ ($1 \leq h \leq r$).
As will be shown in the proof, this is due to the fact that $\rho^{*}(p^{\nu})$
is bounded when $p \nmid D^{*}$, whereas it can be very large when $p|D^{*}$.
It can indeed be as large as the right-hand side of
\eqref{eq:rhostewartbound} as shown by Stewart \cite{stewart}.

Our second theorem gives an order of magnitude instead of an upper bound.

\begin{theorem}
  \label{thm:lowerbound}
  Under the assumptions of Theorem~\ref{thm:mainbound}, and assuming further
  that $Q$ has no fixed prime divisor, $F$ is multiplicative and
  \begin{align}
    \label{eq:Fmino}
    &\phantom{( n_{1},\dotsc,n_{k} \geq 1 )} & F(n_{1},\dotsc,n_{k}) &\gg \eta^{\Omega(n_{1} \dotsc n_{k})} & &( n_{1},\dotsc,n_{k} \geq 1 )
  \end{align}
  for some $\eta \in ]0,1[$, we have
  \begin{align}
    \notag
    \sum_{x < n \leq x+y} &F\big( |Q_{1}(n)| , \dotsc , |Q_{k}(n)| \big) \\
    \label{eq:lowerbound1}
    &\asymp y \prod_{g < p \leq x} \Big( 1 - \frac{\rho(p)}{p} \Big)
    \sum_{n_{1} \dotsm n_{r} \leq x}
    \tilde{F}(n_{1},\dotsc,n_{r})
    \frac{ \hat{\rho}_{\mathbf{R}}(n_{1},\dotsc,n_{r}) }{ [n_{1}\kappa(n_{1}),\dotsc,n_{r}\kappa(n_{r})] } \\
    \label{eq:lowerbound2}
    &\asymp \Delta_{D^{*}} y \prod_{g < p \leq x} \Big( 1 - \frac{\rho(p)}{p} \Big)
    \prod_{\substack{p \leq x \\ p \nmid D^{*}}} \prod_{h = 1}^{r}
    \Big( 1 + \tilde{F}^{(h)}(p) \frac{\rho_{R_{h}}(p)}{p} \Big)
  \end{align}
  where $\Delta_{D^{*}}$ is defined by \eqref{eq:deltaDstar}
  and the implied constant depends at most on $g$, $\alpha$, $\delta$, $A$, $B$, $\eta$.  
\end{theorem}

Thus when $F$ is multiplicative and doesn't take too small values
in the sense above, the bound we obtain in Theorem~\ref{thm:mainbound}
is sharp. The $D^{*}$-dependency of the sum is accurately given by $\Delta_{D^{*}}$
in this case.

Eventually we provide the following result analogous to Theorem~3
of Nair of Tenenbaum \cite{nairtenenbaum}, to illustrate how
the generality of Theorem~\ref{thm:mainbound} can be used.

\begin{theorem}
  Under the assumptions of Theorem~\ref{thm:mainbound}, and provided 
  that $Q(0) \neq 0$, we have
 \begin{multline*}
    \sum_{x < p \leq x+y} F\big(|Q_{1}(p)|,\dotsc,|Q_{k}(p)|\big)
    \ll \frac{Q(0)}{\varphi\big(Q(0)\big)} \Delta_{D^{*}}
    \frac{y}{\log x} \prod_{g < p \leq x} \Big( 1 - \frac{\rho(p)}{p} \Big) \\
    \times \sum_{n_{1} \dotsm n_{r} \leq x}
    \tilde{F}(n_{1},\dotsc,n_{r})
    \frac{ \hat{\rho}_{\mathbf{R}}(n_{1},\dotsc,n_{r}) }{ [n_{1}\kappa(n_{1}),\dotsc,n_{r}\kappa(n_{r})] }
    \text{.}
  \end{multline*}
  where $\Delta_{D^{*}}$ is defined by \eqref{eq:deltaDstar}.
  The dependencies of the various constants are as described in
  Theorem~\ref{thm:mainbound}.
\end{theorem}

We refer to \cite[Proof of Theorem~3]{nairtenenbaum} for the proof of this
Theorem as it is absolutely analogous in our setting.

It is easy to derive the theorems of the introduction from the
previous results. Theorem~\ref{thm:introdiscbound} follows immediately
from Corollary~\ref{cor:discbound} upon observing that when
the $Q_{i}$ are irreducible and $F$ is multiplicative we have
$F=\tilde{F}=G=\tilde{G}$, $k=r$ and $Q_{i}=R_{i}$ for $1 \leq i \leq k$.
Theorem~\ref{thm:introk=1bound} is similarly
derived from Corollary~\ref{cor:Fmultbound}.
We can also recover Theorem~\ref{thm:holowinsky} of Holowinsky with
the refinements mentioned in the introduction by applying
Corollary~\ref{cor:discbound} and its following remark with
$Q_{1}=X$, $Q_{2}=X+\ell$ and $F(n_{1},n_{2})=\lambda_{1}(n_{1})\lambda_{2}(n_{2})$.

The rest of this article is dedicated to proving
Theorems~\ref{thm:mainbound}, \ref{thm:lowerbound}
and Corollaries~\ref{cor:discbound}, \ref{cor:Fmultbound}
which share the same hypotheses (except for
some additional assumptions for Theorem~\ref{thm:lowerbound}).
We therefore place ourselves under the assumptions of
Theorem~\ref{thm:mainbound} for the remaining sections.
We also assume that $F$ is non-zero and further that
$F(1,\dotsc,1)=1$, which is possible upto multipliying
$F$ by a certain constant.
All implicit constants throughout the article will depend
at most on $g$, $\alpha$, $\delta$, $A$, $B$, $\eps$
unless otherwise stated.

\section{Technical lemmas}
\label{sec:technicallemmas}

The purpose of this section is to expose a few technical lemmas inspired
by Lemma~1 and Lemma~2 by Nair and Tenenbaum in \cite{nairtenenbaum}.

We first have to introduce the functions these
lemmas will apply to and their properties.

\begin{lemma}
  \label{lem:FisH}
  Let $\sigma_{1},\dotsc,\sigma_{r}$ be $r$
  positive multiplicative functions satisfying
  ${\sigma_{h}(p^{\nu}) \ll 1}$ uniformly in primes $p$ and
  integers $\nu \geq 1$ ($ 1 \leq h \leq r)$.
  Define
  \begin{align}
    \label{eq:FisH}
    H(n_{1},\dotsc,n_{r}) &\coloneqq \tilde{F}(n_{1},\dotsc,n_{r})
    \frac{ \hat{\rho}_{\mathbf{R}}(n_{1},\dotsc,n_{r}) }{ [n_{1}\kappa(n_{1}),\dotsc,n_{r}\kappa(n_{r})] }
    \sigma_{1}(n_{1}) \dotsm \sigma_{r}(n_{r}), \\
    \notag
    T(n_{1},\dotsc,n_{r}) &\coloneqq \tilde{G}(n_{1},\dotsc,n_{r})
    \frac{ \hat{\rho}_{\mathbf{R}}(n_{1},\dotsc,n_{r}) }{ [n_{1}\kappa(n_{1}),\dotsc,n_{r}\kappa(n_{r})] }
    \sigma_{1}(n_{1}) \dotsm \sigma_{r}(n_{r}) \text{.}
  \end{align}
  We then have
  \begin{equation}
    \label{eq:H1}
    H(a_{1}b_{1},\dotsc,a_{r}b_{r}) \leq T(b_{1},\dotsc,b_{r}) H(a_{1},\dotsc,a_{r})
  \end{equation}
  for all integers $a_{i}$, $b_{j}$ such that $(a_{1} \dotsm a_{r},b_{1} \dotsm b_{r})=1$. We also
  have\footnote{Here and in the sequel the prime next to the
    sum indicates that the sum is over variables which are
    not all zero.}
  \begin{align}
    \label{eq:H2}
    \sideset{}{'}\sum_{\nu_{1},\dotsc,\nu_{r}} T(p^{\nu_{1}},\dotsc,p^{\nu_{r}})
    &\ll \frac{1}{p} \text{,} \\
    \label{eq:H3}
    \sum_{\nu_{1}+\dotsb+\nu_{r} > 2g} T(p^{\nu_{1}},\dotsc,p^{\nu_{r}})
    \cdot p^{\frac{1}{4g}(\nu_{1}+\dotsb+\nu_{r})} &\ll \frac{1}{p^{1+1/4}}
    \text{.}
  \end{align}
\end{lemma}


\begin{proof}
  The inequality \eqref{eq:H1} follows immediately from \eqref{eq:Ftildesubm}
  and the multiplicativity of $\hat{\rho}_{\mathbf{R}}$.
  
  To obtain the two next bounds on $T$, we apply \eqref{eq:rhoRbound},
  \eqref{eq:G3} and the bounds $\sigma_{h}(p^{\nu}) \ll 1$ ($1 \leq h \leq r$) to obtain
  \begin{equation}
    \label{eq:Tmaj}
    T(p^{\nu_{1}},\dotsc,p^{\nu_{r}}) \ll
    \min( A^{g\nu} , B p^{g\eps\nu}) \frac{\rho^{*}(p^{\nu})}{p^{\nu}}
  \end{equation}
  with $\nu = \nu_{1}+\dotsb+\nu_{r}$.
  Using Stewart's bound \eqref{eq:rhostewartbound} on $\rho^{*}$, we obtain
  \begin{equation*}
    \sum_{\nu_{1}+\dotsb+\nu_{r} > 2g} T(p^{\nu_{1}},\dotsc,p^{\nu_{r}}) \cdot p^{\frac{1}{4g}(\nu_{1}+\dotsb+\nu_{r})}
    \ll \sum_{\nu > 2g} p^{\big(g\eps+\frac{1}{4g}-\frac{1}{g}\big) \nu} \nu^{r} \ll \frac{1}{p^{1+c}}
  \end{equation*}
  with $c = \frac{23}{50} \geq \frac{1}{4}$. This proves \eqref{eq:H3}, and to prove \eqref{eq:H2}
  we can now restrict ourselves to the (finite) sum over the $\nu_{i}$ such
  that $\nu_{1}+\dotsb+\nu_{r}\leq 2g$. For these $\nu_{i}$ we have $T(p^{\nu_{1}},\dotsc,p^{\nu_{r}}) \ll \frac{1}{p}$ by
  \eqref{eq:Tmaj} and \eqref{eq:rhotrivialbound}, which concludes
  the proof.
\end{proof}

\begin{lemma}
  \label{lem:lemma1}
  Let $H$ be as in Lemma~\ref{lem:FisH} and let
  $\theta_{1},\dotsc,\theta_{r}$ be $r$ positive multiplicative
  functions satisfying $\theta_{h}(p^{\nu}) = 1 + O(\frac{1}{p})$
  uniformly in primes $p$ and integers $\nu \geq 1$ for all $1 \leq h \leq r$.
  We have, uniformly in $z > 0$,
  \begin{equation*}
    \sum_{n_{1} \dotsm n_{r} \leq z} H(n_{1},\dotsc,n_{r})
    \theta_{1}(n_{1}) \dotsm \theta_{r}(n_{r})
    \ll \sum_{n_{1} \dotsm n_{r} \leq z} H(n_{1},\dotsc,n_{r})
    \text{.}
  \end{equation*}
\end{lemma}

\begin{proof}
  For $1 \leq h \leq r$ and integers $n_{h}$, write
  \begin{equation}
    \label{eq:theta}
      \theta_{h}(n_{h}) = \sum_{d_{h}|n_{h}} \lambda_{h}(d_{h}).
  \end{equation}
  We have $\lambda_{h}(p^{\nu}) = \theta_{h}(p^{\nu}) - \theta_{h}(p^{\nu-1}) \ll \frac{1}{p}$
  for $\nu \geq 1$. For any integers $d_{h}$, $n_{h}$ such
  that $d_{h}|n_{h}$, we can write $n_{h}$ uniquely as $n_{h}=d_{h}t_{h}a_{h}$
  with $t_{h}|d_{h}^{\infty}$ and $(a_{h},d_{h})=1$. Using
  \eqref{eq:H1} and \eqref{eq:theta} we obtain
  \begin{align*}
    \sum_{n_{1} \dotsm n_{r} \leq z} &H(n_{1},\dotsc,n_{r})
    \theta_{1}(n_{1}) \dotsm \theta_{r}(n_{r}) \\
    &\leq \sum_{d_{1},\dotsc,d_{r}} \sum_{\substack{t_{1},\dotsc,t_{r} \\ t_{h}|d_{h}^{\infty}}}
    \lambda_{1}(d_{1}) \dotsm \lambda_{r}(d_{r}) T(d_{1}t_{1},\dotsc,d_{r}t_{r})
    \sum_{a_{1} \dotsm a_{r} \leq z} H(a_{1},\dotsc,a_{r}) \\
    &\leq \Delta_{1} \sum_{a_{1} \dotsm a_{r} \leq z} H(a_{1},\dotsc,a_{r})
  \end{align*}
  where
  \begin{equation*}
    \Delta_{1} = \prod_{p} \Big( 1 + \sideset{}{'}\sum_{s_{1},\dotsc,s_{r}}
    \lambda_{1}(p^{s_{1}}) \dotsm \lambda_{r}(p^{s_{r}})
    \sum_{\ell_{1},\dots,\ell_{r}}
    T(p^{s_{1}+\ell_{1}},\dotsc,p^{s_{r}+\ell_{r}}) \Big)
    \text{.}
  \end{equation*}
  Now by \eqref{eq:H2} and the bound $\lambda_{h}(p^{\nu}) \ll \frac{1}{p}$,
  we have
  \begin{equation*}
    \Delta_{1} = \prod_{p} \bigg( 1 + O\Big( \frac{1}{p^{2}} \Big) \bigg) \ll 1
  \end{equation*}
  which concludes the proof.
\end{proof}

\begin{lemma}
  \label{lem:lemma2}
  Let $H$ be as in Lemma~\ref{lem:FisH}.
  Then for $\chi > 0$, $z \geq e^{4g\chi}$, $\beta = \frac{\chi}{\log z}$,
  \begin{equation*}
    \sum_{P^{+}(n_{1} \dotsm n_{r}) \leq z} H(n_{1},\dotsc,n_{r}) (n_{1} \dotsm n_{r})^{\beta}
    \ll_{\chi} \sum_{P^{+}(n_{1} \dotsm n_{r}) \leq z} H(n_{1},\dotsc,n_{r}) \text{.}
  \end{equation*}
\end{lemma}

\begin{proof}
  For any integer $n$ write $n^{\beta} = \sum_{d|n} \psi(d)$. For any integers $d$, $n$ such
  that $d|n$ we can write $n$ uniquely as $n=dta$, $t|d^{\infty}$, $(a,d)=1$.
  Applying \eqref{eq:H1}, we obtain
  \begin{align*}
    &\sum_{P^{+}(n_{1} \dotsm n_{r}) \leq z} H(n_{1},\dotsc,n_{r}) (n_{1} \dotsm n_{r})^{\beta} \\
      &\leq \sum_{P^{+}(d_{1} \dotsm d_{r}) \leq z} \; \sum_{\substack{t_{1},\dotsc,t_{r} \\ t_{h}|d_{h}^{\infty}}}
      \psi(d_{1}) \dotsm \psi(d_{r}) T(d_{1}t_{1},\dotsc,d_{r}t_{r})
      \sum_{P^{+}(a_{1} \dotsm a_{r}) \leq z} H(a_{1},\dotsc,a_{r}) \\
      &\leq \Delta_{2} \sum_{P^{+}(a_{1} \dotsm a_{r}) \leq z} H(a_{1},\dotsc,a_{r})  
  \end{align*}
  where
  \begin{equation*}
    \Delta_{2} = \prod_{p \leq z} \: \sum_{s_{1},\dotsc,s_{r}}
    \psi(p^{s_{1}}) \dotsm \psi(p^{s_{r}})
    \sum_{\substack{\ell_{1},\dotsc,\ell_{r} \geq 0 \\ \ell_{h} = 0 \text{ if } s_{h}=0}}
    T(p^{s_{1}+\ell_{1}},\dotsc,p^{s_{r}+\ell_{r}}) \text{.}
  \end{equation*}
  We can rewrite this as
  \begin{align*}
    \Delta_{2} &= \prod_{p \leq z} \Big(1 +
    \sideset{}{'}\sum_{\nu_{1},\dotsc,\nu_{r}} T(p^{\nu_{1}},\dotsc,p^{\nu_{r}})
    \prod_{\substack{1 \leq h \leq r \\ \nu_{h} \neq 0}} \:
    \sum_{k = 1}^{\nu_{h}} \psi(p^{k}) \Big) \\
    &= \prod_{p \leq z} \Big( 1 + \sideset{}{'}\sum_{\nu_{1},\dotsc,\nu_{r}}
    T(p^{\nu_{1}},\dotsc,p^{\nu_{r}})
    \prod_{\substack{1 \leq h \leq r \\ \nu_{h} \neq 0}}
    \big( p^{\beta\nu_{h}} - 1 \big) \Big)
    \text{.}
  \end{align*}
  We bound the inner product by distinguishing two cases.
  If $1 \leq \nu_{1}+\dotsb+\nu_{r} \leq 2g$ we have,
  for all $h$, $\beta\nu_{h}\log p \leq 2 \chi g \frac{\log p}{\log z} \ll_{\chi} 1$.
  Therefore for all $h$, $p^{\nu_{h}\beta}-1 \ll_{\chi} \frac{\log p}{\log z}$ which is also $\ll_{\chi} 1$.
  Since at least one $\nu_{h}$ is $\neq 0$ we have
  \begin{equation*}
    \prod_{\substack{1 \leq h \leq r \\ \nu_{h} \neq 0}} \big( p^{\beta\nu_{h}} - 1 \big)
    \ll_{\chi} \frac{\log p}{\log z} \text{.}
  \end{equation*}
  If $\nu_{1}+\dotsb+\nu_{r} > 2g$ we use the trivial bound
  \begin{equation*}
    \prod_{\substack{1 \leq h \leq r \\ \nu_{h} \neq 0}} \big( p^{\beta\nu_{h}} - 1 \big)
    \leq p^{\beta(\nu_{1}+\dotsb+\nu_{r})} \leq p^{\frac{1}{4g}(\nu_{1}+\dotsb+\nu_{r})} \text{.}
  \end{equation*}
  Combining this with our bounds \eqref{eq:H2} and \eqref{eq:H3} on $T$
  we arrive at
  \begin{align*}
    \Delta_{2} &= \prod_{p \leq z} \bigg( 1 + O_{\chi}\Big( \frac{1}{\log z} \frac{\log p}{p}
    + \frac{1}{p^{1+1/4}} \Big) \bigg) \\
    &\leq \exp\bigg( O_{\chi}\Big( \frac{1}{\log z} \sum_{p \leq z} \frac{\log p}{p}
    + \sum_{p \leq z} \frac{1}{p^{1+1/4}} \Big) \bigg)
    \ll_{\chi} 1 \text{.}
  \end{align*}
\end{proof}

\begin{lemma}
  \label{lem:lemma3}
  Let $H$ be as in Lemma~\ref{lem:FisH} and $K > 0$.
  We have, uniformly in $z > 0$,
  \begin{equation*}
    \sum_{P^{+}(n_{1} \dotsm n_{r}) \leq z} H(n_{1},\dotsc,n_{r})
    \leq K^{O(1)} \sum_{P^{+}(n_{1} \dotsm n_{r}) \leq z^{1/K}} H(n_{1},\dotsc,n_{r})
    \text{.}
  \end{equation*}
\end{lemma}

\begin{proof}
  For all $1 \leq h \leq r$ we write $n_{h}=a_{h}b_{h}$ where $P^{+}(a_{h}) \leq z^{\frac{1}{K}}$
  and ${P^{-}(b_{h}) > z^{\frac{1}{K}}}$. Applying \eqref{eq:H1}, we obtain
  \begin{align*}
    \sum_{P^{+}(n_{1} \dotsm n_{r}) \leq z} &H(n_{1},\dotsc,n_{r}) \\
    &\leq \sum_{\substack{P^{+}(b_{1} \dotsm b_{r}) \leq z \\ P^{-}(b_{1} \dotsm b_{r}) > z^{1/K}}}
    \mspace{-8.0mu} T(b_{1},\dotsc,b_{r}) \sum_{P^{+}(a_{1} \dotsm a_{r}) \leq z^{1/K}}
    \mspace{-8.0mu} H(a_{1},\dotsc,a_{r}) \\
    &\leq \bigg( \prod_{z^{1/K} < p \leq z} \:
    \sum_{\nu_{1},\dotsc,\nu_{r}} T(p^{\nu_{1}},\dotsc,p^{\nu_{r}}) \bigg)
    \sum_{P^{+}(a_{1} \dotsm a_{r}) \leq z^{1/K}}
    \mspace{-8.0mu} H(a_{1},\dotsc,a_{r})
    \text{.}
  \end{align*}
  To conclude we observe that by \eqref{eq:H2} the product above is
  \begin{align*}
    \leq \prod_{z^{1/K} < p \leq z} \Big( 1 + \frac{1}{p} \Big)^{O(1)} \leq K^{O(1)}
    \text{.}
  \end{align*}
\end{proof}

\begin{lemma}
  \label{lem:lemma4}
  Let $H$ be as in Lemma~\ref{lem:FisH}. We have, uniformly in $z > 0$,
  \begin{equation}
    \label{eq:lemma4}
    \sum_{P^{+}(n_{1} \dotsm n_{r}) \leq z} H(n_{1},\dotsc,n_{r})
    \asymp \sum_{n_{1} \dotsm n_{r} \leq z} H(n_{1},\dotsc,n_{r})
    \text{.}
  \end{equation}
\end{lemma}

\begin{proof}
  The lower bound is obvious.
  To prove the upper bound, we introduce a constant
  $K > 0$ whose value will be determined later.
  By Lemma~\ref{lem:lemma3}, there exists $L > 0$
  depending on the usual parameters such that
  \begin{equation}
    \label{eq:majolemma4}
    \begin{split}
      \sum_{P^{+}(n_{1} \dotsm n_{r}) \leq z} H(n_{1},\dotsc,n_{r})
      &\leq K^{L} \sum_{P^{+}(n_{1} \dotsm n_{r}) \leq z^{1/K}} H(n_{1},\dotsc,n_{r}) \\
      &\leq U + K^{L} \sum_{n_{1} \dotsm n_{r} \leq z} H(n_{1},\dotsc,n_{r})
    \end{split}
  \end{equation}
  where
  \begin{equation*}
    U = K^{L} \sum_{\substack{n_{1} \dotsm n_{r} > z \\ P^{+}(n_{1} \dotsm n_{r}) \leq z^{1/K}}}
    H(n_{1},\dotsc,n_{r}) \text{.}
  \end{equation*}
  We let $\beta_{K} = \frac{1}{\log (z^{1/K})} = \frac{K}{\log z}$. For $z \geq e^{4gK}$, we have
  
  \begin{align*}
    U &\leq K^{L} z^{-\beta_{K}} \sum_{P^{+}(n_{1} \dotsm n_{r}) \leq z^{1/K}}
    H(n_{1},\dotsc,n_{r}) (n_{1} \dotsm n_{r} )^{\beta_{K}} \\
    &\ll K^{L} e^{-K} \sum_{P^{+}(n_{1} \dotsm n_{r}) \leq z^{1/K}} H(n_{1},\dotsc,n_{r})
  \end{align*}
  where we have used Lemma~\ref{lem:lemma2} with $\chi = 1$
  in the second step.
  For a good choice of $K$ (depending on the usual parameters) we can thus impose
  \begin{equation*}
    U \leq \frac{1}{2} \sum_{P^{+}(n_{1} \dotsm n_{r}) \leq z} H(n_{1},\dotsc,n_{r})
    \text{.}
  \end{equation*}
  Inserting this back into \eqref{eq:majolemma4} yields the desired
  bound for $z$ large enough. When $z$ is bounded, so is the left-hand side
  of \eqref{eq:lemma4} by \eqref{eq:H1}, \eqref{eq:H2} and \eqref{eq:H3}.
  Since the right-hand side is superior to $H(1,\dotsc,1) = 1$,
  \eqref{eq:lemma4} still holds in this case.
\end{proof}

\section{Proof of Theorem~\ref{thm:mainbound}}
\label{sec:mainbound}

In this section we prove Theorem~\ref{thm:mainbound}, following
closely the proof of Theorem~1 in \cite{nairtenenbaum}
with occasional modifications to preserve the uniformity
in the discriminant.

We define
\begin{equation}
  \label{eq:epsilon123}
  \eps_{1} \coloneqq \frac{3}{25} \alpha, \quad
  \eps_{2} \coloneqq \frac{\eps_{1}}{3}, \quad
  \eps_{3} \coloneqq \frac{\eps_{1}}{6g} \text{.}
\end{equation}
Before proceeding to the proof of Theorem~\ref{thm:mainbound}
we establish the following sieve bound, which is
essential to our argument.

\begin{lemma}
  \label{lem:sievebound}
  Let $\Xi$ be the set of fixed prime divisors of $Q$.
  Assume $z \leq x^{\eps_{3}}$ and $a_{1} \dotsm a_{r} \leq x^{\eps_{1}}$. Then for $z$ large enough,
  \begin{equation}
    \label{eq:sievebound}
    \sum_{\substack{x < n \leq x + y \\ a_{h} || R_{h} (n) \; (1 \leq h \leq r) \\ p|Q(n) \Rightarrow p|a_{1} \dotsm a_{r} \\ \mathrm{or}\: p \in \Xi \:\mathrm{or}\: p > z}}
    1 \mspace{20.0mu} \asymp \mspace{20.0mu}
    y \frac{\hat{\rho}_{\mathbf{R}}(a_{1},\dotsc,a_{r})}{ [a_{1}\kappa(a_{1}),\dotsc,a_{r}\kappa(a_{r})] }
    \prod_{\substack{g < p \leq z \\ p \nmid a_{1} \dotsm a_{r} }} \Big( 1 - \frac{\rho(p)}{p} \Big) \text{.}
  \end{equation}
\end{lemma}

\begin{proof}
  We use Brun's sieve as exposed by Halberstam and Richert in
  \cite{sievemethods}, following their notations. We define a sequence
  \begin{equation*}
    \mathcal{A} \coloneqq \{ Q(n) : x < n \leq x + y \text{ such that } a_{h} || R_{h}(n) (1 \leq h \leq r) \}
  \end{equation*}
  and a sifting set of primes
  \begin{equation*}
    \mathcal{B} \coloneqq \{ p \notin \Xi \text{ such that } p \nmid a_{1} \dotsm a_{r} \} \text{.}
  \end{equation*}
  With these definitions the left-hand side of \eqref{eq:sievebound}
  is nothing more than $S(\mathcal{A},\mathcal{B},z)$.
  We have
  \begin{align*}
    X &= y \frac{\hat{\rho}_{\mathbf{R}}(a_{1},\dotsc,a_{r})}{ [a_{1}\kappa(a_{1}),\dotsc,a_{r}\kappa(a_{r})] } \\
    \intertext{and for $d$ squarefree with prime factors in $\mathcal{B}$,}
    \omega(d) &= \rho(d) \text{,} \\
    |R_{d}| &\leq \hat{\rho}_{\mathbf{R}}(a_{1},\dotsc,a_{r}) \rho(d)
    \text{.}
  \end{align*}
  We first check that $X \geq y / (a_{1} \dotsm a_{r})^{2} \geq x^{\alpha - 2\eps_{1}} \geq x^{\frac{19}{25}\alpha} > 1$.
  We also have $\omega(p) \leq g$ and
  \begin{equation*}
    0 \leq \frac{\omega(p)}{p} \leq 1 - \frac{1}{g+1}
  \end{equation*}
  for $p \in \mathcal{B}$, so that $(\Omega_{0})$ holds with $A_{0}=g$
  and $(\Omega_{1})$ holds with $A_{1}=g+1$.
  Lemma~2.2 p.52 of \cite{sievemethods} then implies that $(\Omega_{2}(\kappa))$
  holds with $\kappa = A_{0} = A_{2} = g$.
  The condition $(R)$ is also satisfied in its modified form
  \begin{equation*}
    |R_{d}| \leq L \omega(d)
  \end{equation*}
  with $L = \hat{\rho}_{R}(a_{1},\dotsc,a_{r}) \leq (a_{1} \dotsm a_{r})^{2} \leq x^{2\eps_{1}}$.
  We can therefore apply Theorem~2.1 p.57 of \cite{sievemethods}
  together with its Remark~2, with the choice of parameters
  $b = 1$ and $\lambda = \frac{1}{2e}$. This yields, for $z$ large enough (with respect
  to the $A_{i}$ and $\kappa$, that is with respect to $g$ in our setting),
  \begin{equation*}
    S(\mathcal{A},\mathcal{B},z) = v X W(z) + O( Lz^{24g} )
  \end{equation*}
  where $v \asymp 1$ and
  \begin{equation*}
    W(z) = \prod_{\substack{p \leq z \\ p \in \mathcal{B}}} \Big( 1 - \frac{\omega(p)}{p} \Big) \text{.}
  \end{equation*}
  We have $XW(z) \gg x^{\frac{19}{25}\alpha} (\log x)^{-g}$
  and $Lz^{24g} \ll x^{2\eps_{1} + 24g \eps_{3}} \ll XW(z) x^{-\alpha/5+\eta}$
  for any $\eta > 0$.
  Therefore, for $z$ large enough,
  \begin{equation*}
    S(\mathcal{A},\mathcal{B},z) \asymp X W(z) \text{.}
  \end{equation*}
  To observe that
  \begin{equation*}
    W(z) \asymp \prod_{\substack{g < p \leq x \\ p \nmid a_{1} \dotsm a_{r}}}
    \Big( 1 - \frac{\rho(p)}{p} \Big),
  \end{equation*}
  which stems from the fact that all fixed prime
  divisors $p$ of $Q$ are smaller than $g$.
\end{proof}

We now expose our proof of Theorem~\ref{thm:mainbound}.
Let $x < n \leq x+y$.
We write $Q^{*}(n) = p_{1}^{s_{1}} \dotsm p_{t}^{s_{t}}$
and define $a_{n} = p_{1}^{s_{1}} \dotsm p_{j}^{s_{j}}$ with $j$
maximal so that $a_{n} \leq x^{\eps_{1}}$. We let
$q_{n} = p_{j+1}$ whenever $j \neq t$, else
we let $q_{n}=+\infty$.
We thus have a decomposition
\begin{equation*}
  Q^{*}(n) = a_{n}b_{n}
\end{equation*}
with $P^{+}(a_{n}) < q_{n}$ and $P^{-}(b_{n}) \geq q_{n}$.
Accordingly we decompose the $R_{h}(n)$, ${1 \leq h \leq r}$, in
\begin{equation*}
  R_{h}(n) = a_{hn}b_{hn}
\end{equation*}
with $P^{+}(a_{hn}) < q_{n}$ and $P^{-}(b_{hn}) \geq q_{n}$.
It follows from the definitions above that $a_{n} \leq x^{\eps_{1}}$, $q_{n}=P^{-}(b_{n})$,
$a_{n}||Q(n)$, $a_{hn}||R_{h}(n)$, $a_{n}=a_{1n} \dotsm a_{rn}$ and
$b_{n}=b_{1n} \dotsm b_{rn}$.

We will distinguish five potentially overlapping
classes of integers $x < n \leq x+y$ as follows :

\begin{description}
\item[$(C_1)$] $a_{n} \leq x^{\eps_{1}}$, $P^{-}(b_{n}) > x^{\eps_{3}},$
\item[$(C_2)$] $a_{n} \leq x^{\eps_{2}} ,\;
P^{-}(b_{n}) \leq x^{\eps_{3}} ,\; b_{n} \neq 1,$
\item[$(C_3)$] $x^{\eps_{2}} < a_{n} \leq x^{\eps_{1}} ,\;
\omega < P^{+}(a_{n}) \leq x^{\eps_{3}}$
\item[$(C_4)$] $x^{\eps_{2}} < a_{n} \leq x^{\eps_{1}} ,\;
P^{+}(a_{n}) \leq \omega,$
\item[$(C_5)$] $a_{n} \leq x^{\eps_{2}},\; b_{n} = 1,$
\end{description}
where $\omega$ is a parameter to be chosen later.

For $1 \leq i \leq 5$ we let
\begin{equation*}
  S_{i} = \sum_{n \in (C_{i})} F(|Q_{1}(n)|,\dotsc,|Q_{k}(n)|)
  = \sum_{n \in (C_{i})} \tilde{F}(|R_{1}(n)|,\dotsc,|R_{r}(n)|),
\end{equation*}
the second equality coming from \eqref{eq:switchtoFtilde}.

\textit{Contribution of integers $n \in C_{1}$,
  for which $a_{n} \leq x^{\eps_{1}}$ and $P^{-}(b_{n}) > x^{\eps_{3}}$.}

Since $b_{n} \geq P^{-}(b_{n})^{\Omega(b_{n)}}$ and $\| Q \| \leq x^{\frac{1}{\delta}}$, we have
\begin{equation*}
  \Omega(b_{n}) \leq \frac{\log b_{n}}{\log P^{-}(b_{n})}
  \leq \frac{\log |Q(n)|}{\log P^{-}(b_{n})}
  \leq \Big(g+\frac{1}{\delta}\Big) \frac{1}{\eps_{3}} \text{.}
\end{equation*}
Therefore by \eqref{eq:G1} we have
\begin{equation*}
  \tilde{G}(b_{1n},\dotsc,b_{rn}) \leq A^{g\Omega(b_{n})} \ll 1 \text{.}
\end{equation*}
By \eqref{eq:Ftildesubm} we then obtain that
\begin{align*}
  S_{1} \ll \sum_{a_{1} \dotsm a_{r} \leq x^{\eps_{1}}}
  \tilde{F}(a_{1},\dotsc,a_{r})
  \sum_{\substack{x < n \leq x + y \\ a_{h} || R_{h} (n) (1 \leq h \leq r) \\ p|Q(n) \Rightarrow p|a_{1} \dotsm a_{r} \:\mathrm{or}\: p > x^{\eps_{3}}}} 1
  \text{.}
\end{align*}
Applying Lemma~\ref{lem:sievebound} to bound the inner sum
we obtain
\begin{align*}
  S_{1} \ll y \sum_{a_{1} \dotsm a_{r} \leq x}
  \tilde{F}(a_{1},\dotsc,a_{r})
  \frac{\hat{\rho}_{\mathbf{R}}(a_{1},\dotsc,a_{r})}{[a_{1}\kappa(a_{1}),\dotsc,a_{r}\kappa(a_{r})]}
  \prod_{\substack{g < p \leq x^{\eps_{3}} \\ p \nmid a_{1} \dotsm a_{r} }}
  \Big( 1 - \frac{\rho(p)}{p} \Big) \text{.}
\end{align*}
The inner product is, by \eqref{eq:rhopbound},
\begin{align*}
  &\ll \prod_{h=1}^{r} \prod_{p|a_{h}} \Big( 1 - \frac{1}{p} \Big)^{-g}
  \prod_{x^{\eps_{3}} < p \leq x} \Big( 1 - \frac{1}{p} \Big)^{-g}
  \prod_{g < p \leq x} \Big( 1 - \frac{\rho(p)}{p} \Big) \\
  &\ll \lambda(a_{1}) \dotsm \lambda(a_{r})
  \prod_{g < p \leq x} \Big( 1 - \frac{\rho(p)}{p} \Big)
\end{align*}
where $\lambda(n)=(\frac{n}{\varphi(n)})^{g}$. We deduce that
\begin{equation*}
  S_{1} \ll y \prod_{g < p \leq x} \Big( 1 - \frac{\rho(p)}{p} \Big)
  \sum_{a_{1} \dotsm a_{r} \leq x}
  \tilde{F}(a_{1},\dotsc,a_{r})
  \frac{\hat{\rho}_{\mathbf{R}}(a_{1},\dotsc,a_{r})}{[a_{1}\kappa(a_{1}),\dotsc,a_{r}\kappa(a_{r})]}
  \lambda(a_{1}) \dotsm \lambda(a_{r}) \text{.}
\end{equation*}
Applying Lemmas~\ref{lem:FisH},~\ref{lem:lemma1}
with $\sigma_{h} = 1$, $\theta_{h}=\lambda$ ($1 \leq h \leq r$) to the sum over
the $a_{i}$ in the above we see that $S_{1}$ is of the
expected order of magnitude.

\textit{Contribution of integers $n \in C_{2}$,
  for which $a_{n} \leq x^{\eps_{2}}$, $P^{-}(b_{n}) \leq x^{\eps_{3}}$ and
  $b_{n} \neq 1$.}

Let $q_{n} = P^{-}(b_{n})$. By definition of $a_{n}$ we
have $a_{n}q_{n}^{e_{n}} > x^{\eps_{1}}$ for some $e_{n} \geq 1$. For this
$e_{n}$ we have $q_{n}^{e_{n}} > x^{\eps_{1}-\eps_{2}} = x^{2\eps_{2}}$. We introduce
the minimal integer $f_{n}$ such that $q_{n}^{f_{n}} > x^{2\eps_{2}}$.
Since $q_{n}^{f_{n}-1} \leq x^{2\eps_{2}}$, $q_{n}^{f_{n}} \leq x^{2\eps_{2}+\eps_{3}}$ and in
particular $f_{n} \leq \log x$ and $q_{n}^{f_{n}} \leq y$.

By \eqref{eq:Fsubm} and our assumption $\|Q\| \leq x^{\frac{1}{\delta}}$ we have
\begin{equation}
  \label{eq:Ftrivialbound}
  F(|Q_{1}(n)|,\dotsc,|Q_{k}(n)|) \leq B |Q(n)|^{\eps} \ll x^{(g+\frac{1}{\delta})\eps} \text{.}
\end{equation}
This allows us to bound $S_{2}$ by
\begin{align*}
  S_{2} \ll x^{(g+\frac{1}{\delta})\eps} \sum_{f \leq \log x}
  \sum_{\substack{q \leq x^{\eps_{3}} \\ x^{2\eps_{2}} < q^{f} \leq y}}
  \sum_{\substack{x < n \leq x+y \\ q^{f} | Q^{*}(n)}} 1 \;\text{.}
\end{align*}
The innerest sum is
\begin{equation*}
  \leq \rho^{*}(q^{f}) \Big( \frac{y}{q^{f}} + 1 \Big)
  \ll y \frac{\rho^{*}(q^{f})}{q^{f}} \ll y q^{-\frac{f}{g}}
\end{equation*}
by \eqref{eq:rhostewartbound}. Therefore
\begin{align*}
  S_{2} \ll y x^{(g+\frac{1}{\delta})\eps} \sum_{f \leq \log x}
  \sum_{\substack{q \leq x^{\eps_{3}} \\ x^{-2\eps_{2}} < q^{f}}} q^{-\frac{f}{g}}
  \ll y x^{(g+\frac{1}{\delta})\eps + \eps_{3} - 2\frac{\eps_{2}}{g}} \ll y x^{-c} \log x
\end{align*}
with $c = \frac{\alpha}{25g}$. This is readily seen to be lower than the expected order of
magnitude since the right-hand side of \eqref{eq:mainbound}
is $\gg y (\log x)^{-g}$. This last fact follows from our
assumption $F(1,\dotsc,1)=1$ and \eqref{eq:rhopbound}.

\textit{Contribution of integers $n \in C_{3}$,
  for which $x^{\eps_{2}} < a_{n} \leq x^{\eps_{1}}$ and
$\omega < P^{+}(a_{n}) \leq x^{\eps_{3}}$.}

We define $\ell_{n} \coloneqq P^{+}(a_{n})$.
We write $a_{n} = \ell_{n}^{\nu_{n}} c_{n}$ with $\ell_{n} \nmid c_{n}$ and $a_{nh}=\ell_{n}^{\nu_{hn}}c_{hn}$
with $\ell_{n} \nmid c_{hn}$ ($1 \leq h \leq r$).
By \eqref{eq:G3}, we have
\begin{equation*}
  \tilde{G}(\ell_{n}^{\nu_{1n}},\dotsc,\ell_{n}^{\nu_{rn}}) \leq D(\ell_{n}^{\nu_{n}})
\end{equation*}
where $D$ is the multiplicative function defined by
\begin{align*}
  &\phantom{(\nu \geq 1)} & D(p^{\nu}) &= \min(A^{g\nu},Bp^{g\eps\nu}) & &(\nu \geq 1)
\end{align*}
for primes $p$. We also have, as in the case of the class $(C_{2})$,
\begin{equation*}
  \Omega(b_{n}) \leq (g+\frac{1}{\delta}) \frac{\log x}{\log \ell_{n}}
\end{equation*}
and therefore, upon using \eqref{eq:G1},
\begin{equation*}
  \tilde{G}(b_{1n},\dotsc,b_{rn}) \leq A^{g\Omega(b_{n})} \leq e^{Lu_{n}}
\end{equation*}
with $u_{n} \coloneqq \frac{\log x}{\log \ell_{n}}$ and $L \coloneqq g (g + \frac{1}{\delta}) \log A$.
Note that $L$ depends on the usual parameters.

Applying \eqref{eq:Ftildesubm}, we then obtain
\begin{equation*}
  S_{3} \ll \sum_{\substack{\ell^{\nu} \leq x^{\eps_{1}} \\ w < \ell \leq x^{\eps_{3}}}} D(\ell^{\nu}) e^{Lu}
  \sum_{\substack{ \frac{x^{\eps_{2}}}{\ell^{\nu}} \leq c_{1} \dotsm c_{r} \leq \frac{x^{\eps_{1}}}{\ell^{\nu}} \\
    P^{+}(c_{1} \dotsm c_{r}) < \ell }} 
  \tilde{F}(c_{1},\dotsc,c_{r})
  \sum_{\substack{x < n \leq x + y \\ c_{h} || R_{h} (n) (1 \leq h \leq r)\\ \ell^{\nu} | Q^{*}(n) \\ p |Q(n) \Rightarrow p| \ell c_{1} \dotsm c_{r} \:\mathrm{or}\: p > x^{\eps_{3}}}} 1
\end{equation*}
where $u=\frac{\log x}{\log \ell}$. Now Lemma~\ref{lem:sievebound} can easily be modified to
bound the inner sum above, the new condition $\ell^{\nu}|Q^{*}(n)$
changing the right-hand side of \eqref{eq:sievebound} upto a factor $\frac{\rho^{*}(\ell^{\nu})}{\ell^{\nu}}$.
We thereby obtain
\begin{multline*}
  S_{3} \ll y \sum_{w < \ell \leq x} \sum_{\nu}
  D(\ell^{\nu}) \frac{\rho^{*}(\ell^{\nu})}{\ell^{\nu}} e^{Lu} \\
  \times \sum_{\substack{ \frac{x^{\eps_{2}}}{\ell^{\nu}} \leq c_{1} \dotsm c_{r} \\
    P^{+}(c_{1} \dotsm c_{r}) < \ell }}
  \tilde{F}(c_{1},\dotsc,c_{r}) \frac{\hat{\rho}_{\mathbf{R}}(c_{1},\dotsc,c_{r})}{[c_{1}\kappa(c_{1}),\dotsc,c_{r}\kappa(c_{r})]}
  \prod_{\substack{g < p \leq q \\ p \nmid \ell c_{1} \dotsm c_{r} }} \Big( 1 - \frac{\rho(p)}{p} \Big) \text{.}
\end{multline*}
The inner product is, by \eqref{eq:rhopbound},
\begin{align*}
  &\ll \prod_{h=1}^{r} \prod_{p|c_{h}} \Big( 1 - \frac{1}{p} \Big)^{-g}
  \prod_{q < p \leq x} \Big( 1 - \frac{1}{p} \Big)^{-g}
  \prod_{g < p \leq x} \Big( 1 - \frac{\rho(p)}{p} \Big) \\
  &\ll \lambda(c_{1}) \dotsm \lambda(c_{r}) u^{g}
  \prod_{g < p \leq x} \Big( 1 - \frac{\rho(p)}{p} \Big)
\end{align*}
where $\lambda(n)=(\frac{n}{\varphi(n)})^{g}$.
Letting $\chi=\frac{3L}{\eps_{2}}$ and $\beta = \frac{\chi}{\log \ell}$, we therefore have
\begin{multline}
  \label{eq:S3beforelemmas}
    S_{3} \ll y \prod_{g < p \leq x} \Big( 1 - \frac{\rho(p)}{p} \Big)
    \sum_{w < \ell \leq x} \sum_{\nu}
    D(\ell^{\nu}) \frac{\rho^{*}(\ell^{\nu})}{\ell^{\nu}}
    e^{Lu} u^{g} \Big(\frac{x^{\eps_{2}}}{\ell^{\nu}}\Big)^{-\beta} \\
    \times \sum_{P^{+}(c_{1} \dotsm c_{r}) < \ell}
    F(c_{1},\dotsc,c_{r}) \frac{\hat{\rho}_{\mathbf{R}}(c_{1},\dotsc,c_{r})}{[c_{1}\kappa(c_{1}),\dotsc,c_{r}\kappa(c_{r})]}
    \lambda(c_{1}) \dotsm \lambda(c_{r})
    (c_{1} \dotsm c_{r})^{\beta} \text{.}
\end{multline}
We now remark that
\begin{equation*}
  e^{Lu}u^{g} \Big(\frac{x^{\eps_{2}}}{\ell^{\nu}}\Big)^{-\beta} =
  e^{-2Lu} u^{g} e^{\nu\chi} \ll e^{-Lu} e^{\nu\chi},
\end{equation*}
and we apply Lemmas~\ref{lem:FisH},
\ref{lem:lemma2}, \ref{lem:lemma4} with $\sigma_{h} = \lambda$ ($1 \leq h \leq r$)
to the sum over the $c_{i}$ in \eqref{eq:S3beforelemmas} to obtain
\begin{equation}
  \label{eq:S3}
  \begin{split}
    S_{3} \ll y \prod_{g < p \leq x} \Big( 1 - &\frac{\rho(p)}{p} \Big)
    \bigg( \sum_{w < \ell \leq x} e^{-Lu}
    \sum_{\nu} e^{\nu\chi} D(\ell^{\nu})
    \frac{\rho^{*}(\ell^{\nu})}{\ell^{\nu}} \bigg) \\
    & \times \sum_{c_{1} \dotsm c_{r} \leq x}
    \tilde{F}(c_{1},\dotsc,c_{r})
    \frac{\hat{\rho}_{\mathbf{R}}(c_{1},\dotsc,c_{r})}{[c_{1}\kappa(c_{1}),\dotsc,c_{r}\kappa(c_{r})]}
    \lambda(c_{1}) \dotsm \lambda(c_{r})
    \text{.}
  \end{split}
\end{equation}
Using \eqref{eq:rhotrivialbound} and \eqref{eq:rhostewartbound}
we see that, taking $\omega = e^{2g\chi}$,
\begin{equation*}
  \sum_{\nu} e^{\nu\chi} D(\ell^{\nu}) \frac{\rho^{*}(\ell^{\nu})}{\ell^{\nu}}
  \ll \frac{1}{\ell} + \sum_{\nu > 2g} e^{\nu\chi} \ell^{g\eps\nu} \ell^{-\frac{\nu}{g}}
  \ll \frac{1}{\ell} \text{.}
\end{equation*}
Also by integration by parts we have
\begin{equation*}
  \sum_{\ell \leq x} \frac{e^{-Lu}}{\ell} \ll 1 \text{.}
\end{equation*}
The sum over $\ell$ in \eqref{eq:S3} is therefore bounded.
Applying Lemmas~\ref{lem:FisH}, \ref{lem:lemma1} with $\sigma_{h}=1$,
$\theta_{h}=\lambda$ ($1 \leq h \leq r$) to the sum
over the $c_{i}$ in \eqref{eq:S3} we thus see that
\eqref{eq:S3} is compatible with \eqref{eq:mainbound}.

\textit{Contribution of integers $n \in C_{4}$,
  for which $x^{\eps_{2}} < a_{n} \leq x^{\eps_{1}}$ and
$P^{+}(a_{n}) \leq \omega$.}

We use the trivial bound \eqref{eq:Ftrivialbound} to obtain
\begin{align}
  \notag
  S_{4} &\ll x^{(g+\frac{1}{\delta})\eps}
  \sum_{\substack{x^{\eps_{2}} < a \leq x^{\eps_{1}} \\ P^{+}(a) \leq \omega}}
  \sum_{\substack{x < n \leq x+y \\ a | Q^{*}(n)}} 1 \\
  \label{eq:S4}
  &\ll y x^{(g+\frac{1}{\delta})\eps}
  \sum_{\substack{x^{\eps_{2}} < a \leq x^{\eps_{1}} \\ P^{+}(a) \leq \omega}}
  \frac{\rho^{*}(a)}{a}
  \text{.}
\end{align}

For integers $a$ such that $P^{+}(a) \leq \omega$
we have $\omega(a) \leq \pi(\omega) \ll 1$ and
hence, by \eqref{eq:rhostewartbound},
$\rho^{*}(a) \leq g^{\omega(a)}a^{1-\frac{1}{g}} \ll a^{1-\frac{1}{g}}$.
Inserting this bound in \eqref{eq:S4} we obtain
\begin{equation*}
  S_{4} \ll y x^{(g+\frac{1}{\delta})\eps}
  \sum_{\substack{x^{\eps_{2}} < a \leq x^{\eps_{1}} \\ P^{+}(a) \leq \omega}} a^{-\frac{1}{g}}
  \ll y x^{(g+\frac{1}{\delta})\eps - \frac{\eps_{2}}{g}} (\log x)^{\omega}
  \ll y x^{-c} (\log x)^{\omega}
\end{equation*}
with $c=\frac{\alpha}{25g}$. This is compatible with \eqref{eq:mainbound} as argued in the case
of integers $n \in (C_{2})$.

\textit{Contribution of integers $n \in C_{5}$,
  for which $a_{n} \leq x^{\eps_{2}}$ and
$b_{n} = 1$.}

We use the trival bound \eqref{eq:Ftrivialbound} to obtain
\begin{equation*}
  S_{5} \ll x^{(g+\frac{1}{\delta})\eps} \sum_{a \leq x^{\eps_{2}}} \; \sum_{\substack{x < n \leq x+y \\ Q^{*}(n) = a}} 1
  \ll x^{(g+\frac{1}{\delta})\eps + \eps_{2}} \ll y x^{-c} \text{.}
\end{equation*}
with $c=-\frac{47}{50}\alpha$. This is compatible with \eqref{eq:mainbound} as argued in the case
of integers $n \in (C_{2})$.

\section{Proof of Corollaries \ref{cor:discbound} and \ref{cor:Fmultbound}}
\label{sec:corollaries}

To derive Corollaries~\ref{cor:discbound} and \ref{cor:Fmultbound}
from Theorem~\ref{thm:mainbound} we focus on the sum
\begin{equation*}
  \tilde{S} = \sum_{n_{1} \dotsm n_{r} \leq x} \tilde{F}(n_{1},\dotsc,n_{r})
  \frac{ \hat{\rho}_{\mathbf{R}}(n_{1},\dotsc,n_{r}) }{ [n_{1}\kappa(n_{1}),\dotsc,n_{r}\kappa(n_{r})] }
\end{equation*}
appearing in the right-hand side of \eqref{eq:mainbound}.
We shall establish upper bounds for $\tilde{S}$ as well as lower bounds
that we need for the proof of Theorem~\ref{thm:lowerbound}.
Corollary~\ref{cor:discbound} is
a direct consequence of the following Lemma.

\begin{lemma}
  \label{lem:littlediscbound}
  We have
  \begin{equation}
    \label{eq:littlediscbound}
    \tilde{S} \ll \Delta_{D^{*}}
    \sum_{\substack{ a_{1} \dotsm a_{r} \leq x  \\ (a_{1} \dotsm a_{r},D^{*}) =1 \\ (a_{i},a_{j}) = 1 \: (i \neq j)}}
    \tilde{F}(a_{1},\dotsc,a_{r})
    \frac{\rho_{R_{1}}(a_{1}) \dotsm \rho_{R_{r}}(a_{r})}{a_{1} \dotsm a_{r}}.
  \end{equation}
\end{lemma}

\begin{proof}
  For all $1 \leq h \leq r$, we write $n_{h}=d_{h}a_{h}$
  with $d_{h}|D^{*\infty}$ and $(a_{h},D^{*})=1$.
  By \eqref{eq:Ftildesubm} and the submultiplicativity of $\tilde{G}$,
  we then have
  \begin{equation}
    \label{eq:prediscbound}
    \tilde{S} \ll \Delta_{4}
    \sum_{\substack{a_{1} \dotsm a_{r} \leq x \\ (a_{1} \dotsm a_{r},D^{*})=1}}
    \tilde{F}(a_{1},\dotsc,a_{r})
    \frac{ \hat{\rho}_{\mathbf{R}}(a_{1},\dotsc,a_{r}) }{ [a_{1}\kappa(a_{1}),\dotsc,a_{r}\kappa(a_{r})] }
  \end{equation}
  where
  \begin{equation}
    \label{eq:delta4}
    \Delta_{4} = \prod_{p | D^{*}} \bigg( 1 +
    \sideset{}{'}\sum_{\nu_{1},\dotsc,\nu_{r}}
    \tilde{G}(p^{\nu_{1}},\dotsc,p^{\nu_{r}})
    \frac{ \hat{\rho}_{\mathbf{R}}(p^{\nu_{1}},\dotsc,p^{\nu_{r}}) }{ p^{\max_{h}(\nu_{h})+1} }
    \bigg) \text{.}
  \end{equation}
  Now let $1 \leq m \leq r$ and define $\mu_{m}=\deg(R_{m})$.
  By the definition \eqref{eq:rhoR} of $\hat{\rho}_{\mathbf{R}}$ and \eqref{eq:rhohstewartbound},
  we have
  \begin{equation*}
    \frac{ \hat{\rho}_{\mathbf{R}}(p^{\nu_{1}},\dotsc,p^{\nu_{r}}) }{ p^{\max_{h}(\nu_{h})} }
    \leq \frac{\rho_{R_{m}}(p^{\nu_{m}})}{p^{\nu_{m}}} \leq \mu_{m} p^{-\nu_{m}/\mu_{m}} \text{.}
  \end{equation*}
  Using this bound and \eqref{eq:G1}, we obtain
  \begin{equation*}
    \sum_{\substack{\nu_{m} > \mu_{m} \\ \nu_{1}+\dotsb+\nu_{r} \leq 2g}} \tilde{G}(p^{\nu_{1}},\dotsc,p^{\nu_{r}})
    \frac{ \hat{\rho}_{\mathbf{R}}(p^{\nu_{1}},\dotsc,p^{\nu_{r}}) }{ p^{\max_{h}(\nu_{h})+1} }
    \ll \sum_{\nu > \mu_{m}} p^{-\nu/\mu_{m}} \ll \frac{1}{p^{1+1/\nu_{m}}} \text{.}
  \end{equation*}
  Since this is true for all $1 \leq m \leq r$ and since a
  similar bound holds for the tail $\sum_{\nu_{1}+\dotsb+\nu_{r}>2g}$ by \eqref{eq:H3},
  it follows that $\Delta_{4} \asymp \Delta_{D^{*}}$, where $\Delta_{D^{*}}$ is defined
  by \eqref{eq:deltaDstar}.
  
  It remains to rewrite the sum over the $a_{i}$ in \eqref{eq:prediscbound}.
  To this end we use certain algebraic facts about the discriminant
  and the resultant, the proof of which can be found in e.g. \cite{lang}.
  For $h \neq i$ there exists polynomials $U$, $V$ in $\Z[X]$
  such that
  \begin{equation}
    \label{eq:resultant}
    R_{h}(X)U(X) + R_{i}(X)V(X) = \Res(R_{h},R_{i})
  \end{equation}
  where $\Res(R_{h},R_{i})$ is the resultant of $R_{h}$ and $R_{i}$.
  When $\hat{\rho}_{\mathbf{R}}(a_{1},\dotsc,a_{r})$ is non-zero there exists
  an integer $n$ such that $a_{i}|R_{i}(n)$ and $a_{j}|R_{j}(n)$.
  Taking $X=n$ in \eqref{eq:resultant} we then see that $(a_{i},a_{h})|\Res(R_{h},R_{i})$.
  Since $\Disc(R_{h}R_{i})=\Res(R_{h},R_{i})^{2}\Disc(R_{h})\Disc(R_{i})$
  and $\Disc(R_{h}R_{i}) | \Disc(Q^{*}) = D^{*}$, we have
  that $\Res(R_{h},R_{i}) | D^{*}$.
  Therefore $(a_{i},a_{h})|D^{*}$, and since the $a_{j}$ are coprime to
  $D^{*}$ we have further $(a_{i},a_{h})=1$.
  We deduce that $\hat{\rho}_{\mathbf{R}}(a_{1},\dotsc,a_{r})$ is zero
  unless the $a_{i}$ are mutually coprime in which case
  we have, by multiplicativity of $\hat{\rho}_{\mathbf{R}}$,
  \begin{align*}
    \hat{\rho}_{\mathbf{R}}(a_{1},\dotsc,a_{r})
    &= \prod_{h=1}^{r} \hat{\rho}_{\mathbf{R}}^{(h)}(a_{h}) \\
    &= \prod_{h=1}^{r} \prod_{p^{\nu} || a_{h}}
    \big( \rho_{R_{h}}(p^{\nu_{h}})p - \rho_{R_{h}}(p^{\nu_{h}+1}) \big) \\
    &\leq \prod_{h=1}^{r} \rho_{R_{h}}(a_{h}) \kappa(a_{h}) \text{.}
  \end{align*}
  Inserting this back in \eqref{eq:prediscbound} we recover
  \eqref{eq:littlediscbound}.
\end{proof}

Corollary~\ref{cor:Fmultbound} is obtained in a similar fashion,
by applying the following Lemma.

\begin{lemma}
  \label{lem:littleFmultbound}
  We have
  \begin{equation*}
    \tilde{S} \ll \Delta_{D^{*}}
    \prod_{\substack{p \leq x \\ p \nmid D^{*}}} \prod_{h = 1}^{r}
    \Big( 1 + \tilde{G}^{(h)}(p) \frac{\rho_{R_{h}}(p)}{p} \Big) \text{.}
  \end{equation*}
  The right-hand side in the above is also a lower bound for $\tilde{S}$
  when $F$ is assumed to be multiplicative.
\end{lemma}

\begin{proof}
  Applying Lemmas~\ref{lem:FisH}, \ref{lem:lemma3}, \ref{lem:lemma4}
  with $\sigma_{h}=1$ ($1 \leq h \leq r$), we obtain
  \begin{equation*}
    \tilde{S} \asymp \sum_{P^{+}(n_{1} \dotsm n_{r}) \leq x^{(2g-2)/\delta}} \tilde{F}(n_{1},\dotsc,n_{r})
    \frac{ \hat{\rho}_{\mathbf{R}}(n_{1},\dotsc,n_{r}) }{ [n_{1}\kappa(n_{1}),\dotsc,n_{r}\kappa(n_{r})] }
    \text{.}
  \end{equation*}
  We have $D^{*} \ll \| Q^{*} \|^{2g-2}$.
  Since $Q^{*}|Q$, we have $\|Q^{*}\| \leq C\|Q\|$
  where $C$ depends on $g$ at most (see e.g. \cite{pritsker} for precise
  results on the norm of a factor of a polynomial).
  By our assumption $x \geq \| Q \|^{\delta}$
  we therefore have $D^{*} \leq x^{(2g-2)/\delta}$
  for $x$ large enough with respect to
  the usual parameters.
  Using this fact and $\tilde{F} \leq \tilde{G}$ and
  the submultiplicativity of $\tilde{G}$ we can write
  \begin{equation}
    \label{eq:prelittlediscbound}
    \tilde{S} \ll \Delta_{4}
    \prod_{\substack{p \leq x^{(2g-2)/\delta} \\ p \nmid D^{*}}}
    \bigg( 1 + \sideset{}{'}\sum_{\nu_{1},\dotsc,\nu_{r}}
    \tilde{G}(p^{\nu_{1}},\dotsc,p^{\nu_{r}})
    \frac{\hat{\rho}_{\mathbf{R}}(p^{\nu_{1}},\dotsc,p^{\nu_{r}})}{p^{\max_{h}(\nu_{h})+1}}
    \bigg) \text{.}
  \end{equation}
  where $\Delta_{4}$ is defined by \eqref{eq:delta4} as previously, and 
  has been proven to be $\asymp \Delta_{D^{*}}$. When $F$ is multiplicative,
  so is $\tilde{F}=\tilde{G}$, and the right-hand side of
  \eqref{eq:prelittlediscbound} is therefore also a lower bound for $\tilde{S}$.
  
  Now by \eqref{eq:H2} the main term of the product in
  \eqref{eq:prelittlediscbound} is $1 +  O(\frac{1}{p})$ and we
  can thus restrict the product to primes $p \leq x$.
  By \eqref{eq:H3} we can also restrict the inner sum in
  \eqref{eq:prelittlediscbound} to variables $\nu_{i}$
  satisfying $\nu_{1}+\dotsb+\nu_{r} \leq 2g$.
  For those values we have, by \eqref{eq:rhoRbound},
  \eqref{eq:rhocoprimetodiscbound} and \eqref{eq:G1},
  \begin{equation*}
    \tilde{G}(p^{\nu_{1}},\dotsc,p^{\nu_{r}})
    \frac{\hat{\rho}_{\mathbf{R}}(p^{\nu_{1}},\dotsc,p^{\nu_{r}})}{p^{\max_{h}(\nu_{h})+1}}
    \leq A^{2g^{2}} \frac{g^{*}}{p^{\nu_{1}+\dotsb+\nu_{r}}} \text{.}
  \end{equation*}
  We can therefore further restrict the inner sum in
  \eqref{eq:prelittlediscbound} to variables $\nu_{i}$ satisfying the condition
  $\nu_{1}+\dotsb+\nu_{r} \leq 1$. The Lemma then easily follows.
\end{proof}
  
\section{Proof of Theorem~\ref{thm:lowerbound}}
\label{sec:lowerbound}

The purpose of this section is to prove Theorem~\ref{thm:lowerbound}.
The upper bounds follow immediately from Theorem~\ref{thm:mainbound} and
Corollary~\ref{cor:Fmultbound}, we are therefore only concerned
with proving the lower bounds.

In this section we assume that the requirements of
Theorem~\ref{thm:lowerbound} are fullfilled.
We also now allow implicit constants to depend on the paramater
$\eta < 1$ on top of the usual parameters.
We retain the definitions \eqref{eq:epsilon123} of $\eps_{1}$, $\eps_{2}$ and $\eps_{3}$.

We let
\begin{equation*}
  S = \sum_{x < n \leq x+y} F\big(|Q_{1}(n)|,\dotsc,|Q_{k}(n)|\big) \text{.}
\end{equation*}
For an integer $n$ we write
\begin{align*}
  Q(n) = a_{n}b_{n}, \quad R_{h}(n) = a_{hn}b_{hn} \quad (1 \leq h \leq r)
\end{align*}
with $P^{+}(a_{n}) < x^{\eps_{3}}$, $P^{-}(b_{n}) \geq x^{\eps_{3}}$,
$P^{+}(a_{hn}) < x^{\eps_{3}}$ and $P^{-}(b_{hn}) \geq x^{\eps_{3}}$.

Since $b_{n} \geq P^{-}(b_{n})^{\Omega(b_{n)}}$ and $\|Q\| \leq x^{\frac{1}{\delta}}$ we have
\begin{equation*}
  \Omega(b_{n}) \leq \frac{\log b_{n}}{\log P^{-}(b_{n})}
  \leq \frac{\log |Q(n)|}{\log P^{-}(b_{n})}
  \leq \Big(g+\frac{1}{\delta}\Big) \frac{1}{\eps_{3}} \text{.}
\end{equation*}
By \eqref{eq:Fmino} we then have
\begin{equation*}
  \tilde{F}(b_{1n},\dotsc,b_{rn}) \geq \eta^{\Omega(b_{n})} \gg 1 \text{.}
\end{equation*}
Keeping only the integers $n$ such that $a_{1n} \dotsm a_{rn} \leq x^{\eps_{1}}$,
we obtain, by multiplicativity of $F$ and the above bound,
\begin{equation*}
  S \gg  \sum_{a_{1} \dotsm a_{r} \leq x^{\eps_{1}}} \tilde{F}(a_{1},\dotsc,a_{r})
  \sum_{\substack{x < n \leq x + y \\ a_{h} || R_{h} (n) \; (1 \leq h \leq r) \\ p|Q(n) \Rightarrow p|a_{1} \dotsm a_{r} \:\mathrm{or}\: p > x^{\eps_{3}}}}
  1 \text{.}
\end{equation*}
The inner sum can be estimated by applying Lemma~\ref{lem:sievebound},
using the fact that $Q$ has no fixed prime divisor. This yields,
as in the proof of Theorem~\ref{thm:mainbound},
\begin{equation*}
  S \gg y \prod_{g < p \leq x} \Big( 1 - \frac{\rho(p)}{p} \Big)
  \sum_{a_{1} \dotsm a_{r} \leq x^{\eps_{1}}}
  \tilde{F}(a_{1},\dotsc,a_{r})
  \frac{\hat{\rho}_{\mathbf{R}}(a_{1},\dotsc,a_{r})}{[a_{1}\kappa(a_{1}),\dotsc,a_{r}\kappa(a_{r})]}
  \lambda(a_{1}) \dotsc \lambda(a_{r})
\end{equation*}
where $\lambda(n)=(\frac{\varphi(n)}{n})^{g}$.
Applying Lemmas~\ref{lem:FisH}, \ref{lem:lemma1} with $\sigma_{h}=\lambda$,
$\theta_{h}=\lambda^{-1}$ ($1 \leq h \leq r$) to the sum over the $a_{i}$
in the above we obtain
\begin{equation*}
  S \gg y \prod_{g < p \leq x} \Big( 1 - \frac{\rho(p)}{p} \Big)
  \sum_{a_{1} \dotsm a_{r} \leq x^{\eps_{1}}}
  \tilde{F}(a_{1},\dotsc,a_{r})
  \frac{\hat{\rho}_{\mathbf{R}}(a_{1},\dotsc,a_{r})}{[a_{1}\kappa(a_{1}),\dotsc,a_{r}\kappa(a_{r})]}
  \text{.}
\end{equation*}
Further applying Lemmas~\ref{lem:FisH}, \ref{lem:lemma3}, \ref{lem:lemma4}
with $\sigma_{h}=1$ to the sum over the $a_{i}$ we recover the lower bound
in \eqref{eq:lowerbound1}.
The lower bound in \eqref{eq:lowerbound2} then
follows from Lemma~\ref{lem:littleFmultbound}.

\vspace{2.0em}

Kevin Henriot

Université de Montréal

Département de Mathématiques et de Statistique,

Pavillon André-Aisenstadt, bureau 5190,

2900 Édouard-Montpetit

Montréal, Québec, Canada, H3C 3J7

email: \texttt{henriot@dms.umontreal.ca}

\end{document}